\newif\ifabstr
\abstrfalse 
\documentclass[11pt, a4paper]{article}
\usepackage{fullpage}

\usepackage[latin1]{inputenc}

\usepackage{amsmath}
\usepackage{amssymb,amsthm,graphicx}
\usepackage[dvips]{epsfig}
\usepackage{amsfonts}
\usepackage{xcolor}
\usepackage{hyperref}
\usepackage{enumerate}
\usepackage[normalem]{ulem} %Crossing out the text during editations
\usepackage{authblk}
\newtheorem{theorem}{Theorem}

\theoremstyle{remark}

\theoremstyle{definition}

\definecolor{darkgreen}{rgb}{0,.5,0}

\newcounter{sideremark}

\DeclareMathOperator{\sd}{sd}

\DeclareMathOperator{\wsat}{wsat}
\title{A note on the computational complexity of weak saturation
}
\ifabstr
\author{}
\else
\author{Martin Tancer}

\author{Mykhaylo Tyomkyn\thanks{Both authors were supported by GA\v{C}R grant 25-17377S. 
}}

\affil{Department of Applied Mathematics, Faculty of Mathematics and Physics,
Charles~University, Prague, Czech Republic}
\fi

\date{}

\begin{document}

\maketitle

\begin{abstract}
We prove that determining the weak saturation number of a host graph $F$ with respect to a pattern graph $H$ is already a computationally hard problem when $H$ is the triangle. As our main tool we establish a connection between weak saturation and shellability of simplicial complexes.
\end{abstract}

\section{Introduction} 

\paragraph{Weak saturation.}
Let $F$ and $H$ be graphs and let $G$ be a spanning subgraph of $F$. We say
that $G$ is \emph{weakly $H$-saturated in $F$}, if the edges of $E(F) \setminus
E(G)$ can be ordered into a sequence $e_1, \dots, e_m$ in such a way that for
every $i \in [m]$ the graph obtained from $G$ by adding the edges $e_1, \dots,
e_i$ contains a copy of $H$ which contains $e_i$. The \emph{weak saturation
number} $\wsat(F,H)$ is defined as the minimum number of edges of a graph which
is weakly $H$-saturated in $F$.

%Computing weak saturation numbers for specific classes of graphs (or
%hypergraphs) has been subject to an extensive research 

The concept of weak saturation was first introduced in 1968 by Bollob\'as~\cite{BB68} who considered the case when $F$ and $H$ are complete graphs, and conjectured that $\wsat(K_n,K_t)=\binom{n}{2}-\binom{n-t+2}{2}$. This was confirmed by Frankl~\cite{Frankl82}, and, independently, Kalai~\cite{Kal84, Kal85} (a version for matroids was proven earlier by Lov\'asz~\cite{Lovasz77}), and extended by Alon~\cite{Alon} and Blokhuis~\cite{blokhuis1990solution}. Subsequently, weak saturation was studied for different classes of graphs and (in the analogous setting) hypergraphs, see~\cite{Alon, BBMR, Balogh98, btt21,
	Erdos91, Faudree14,  Morrison18, MoshkShap, Pikhurko01a,  Pikhurko01b, ShTy, TZ1, TZ2, Tuza88, Tuza92}. As a common theme,  upper bounds on $wsat(F,H)$ are usually established via simple constructions, while proving lower bounds tends to be much harder, and typically requires methods from algebra or geometry. 

In this note we show that any kind of classification of weak saturation numbers
in full generality is hopeless unless P = NP. More concretely, we show that
determining the weak saturation number is already hard in a seemingly very
simple case when $H$ is the complete graph on three vertices $K_3$:

\begin{theorem}
\label{t:main}
  Given on input a graph $F$ with $n$ vertices, it is NP-hard to decide whether $\wsat(F, K_3) =
  n-1$.
\end{theorem}

Note that $\wsat(F, K_3) \geq n-1$ for any connected $n$-vertex graph $F$, as any
weakly $K_3$-saturated graph in $F$ must be spanning. In our reduction $F$ will always be connected.

Our main tool to prove Theorem~\ref{t:main} is to establish a connection
between weak saturation and shellability (and collapsibility) of simplicial
complexes. We point out that a recent preprint~\cite{chakraborti-cho-kim-kim24}
establishes a connection between weak saturation and $d$-collapsibility
(closely related to collapsibility) in the context of fractional Helly type
theorems. Our setting, however, is quite different.

\paragraph{Simplicial complexes, shellability and collapsibility.}
Let us recall that a(n abstract) \emph{simplicial complex} is a set
system $K$ such that if $\sigma \in K$ and $\tau \subseteq \sigma$, then $\tau \in K$. We will
consider only finite simplicial complexes. The set of \emph{vertices} of $K$ is
the set $\bigcup K$. The elements of $K$ are called \emph{faces} of $K$ and the
\emph{dimension} of a face $\sigma \in K$ is defined as $\dim \sigma := |\sigma| - 1$. The
faces of dimension $1$ are \emph{edges} and the faces of dimension $2$ are
\emph{triangles}. (A triangle in a simplicial complex should not be confused
with a graph-theoretic triangle, a copy of $K_3$. We will avoid using the
notion ``triangle'' in the latter context.) The \emph{dimension} of the complex,
$\dim K$ is defined as the maximum of the dimensions of faces in $K$.
Given a non-negative integer $k$, the
\emph{$k$-skeleton} of a simplicial complex $K$ is a subcomplex of $K$ denoted
$K^{(k)}$ consisting of faces of $K$ of dimension at most $k$. From now on we
regard graphs as (at most) $1$-dimensional simplicial complexes. In particular, the $1$-skeleton of a simplicial complex is a graph.

Given a sequence
$\vartheta_1, \dots, \vartheta_k$ of faces of $K$ we denote by $K[\vartheta_1,
\dots, \vartheta_k]$ the subcomplex of $K$ \emph{induced} by these faces, that
is, the subcomplex formed by faces $\sigma$ such that $\sigma \subseteq
\vartheta_i$ for some $i \in [k]$. An inclusion maximal face of a
simplicial complex is a \emph{facet} and a simplicial complex is \emph{pure} if
all facets have the same dimension.
A pure $d$-dimensional complex $K$ is \emph{shellable} if there is an ordering
$\vartheta_1, \dots, \vartheta_m$ of all facets of $K$ such that for every $i
\in \{2, \dots, m\}$ the complex $K[\vartheta_i] \cap K[\vartheta_1, \dots,
\vartheta_{i-1}]$ is pure and $(d-1)$-dimensional.\footnote{In other words,
$\vartheta_i$ meets the previous facets in a pure $(d-1)$-dimensional subcomplex.}

A simplicial complex $K'$ arises from $K$ by an \emph{elementary collapse} if
there is a face $\tau$ of $K$ contained in a single facet $\sigma$ distinct
from $\tau$ and $K'$ is obtained from $K$ by removing all faces containing
$\tau$. A simplicial complex $K$ \emph{collapses} to a subcomplex $L$, if there is a sequence
$K = K_1, K_2, \dots, K_\ell = L$ of simplicial complexes such that $K_{i+1}$ is
obtained from $K_i$ by an elementary collapse for $i \in [\ell - 1]$. A
simplicial complex $K$ is \emph{collapsible} if it collapses to a point 
(an arbitrary vertex of $K$).

The \emph{reduced Euler characteristic} of a complex $K$ is defined as
\[
  \tilde \chi (K) = \sum_{i=-1}^{\dim K} (-1)^i f_i(K)
\]
where $f_i(K)$ is the number of $i$-dimensional faces of $K$. (Note that the
empty set has the dimension equal to $-1$.) Given a simplicial complex $K$, its
\emph{barycentric subdivision} $\sd K$ is a complex whose vertices are nonempty faces
of $K$ and whose faces are collections $\{\vartheta_1, \dots, \vartheta_k\}$ of
faces of $K$ with $\emptyset \neq \vartheta_1 \subsetneq \vartheta_2 \subsetneq
\cdots \subsetneq \vartheta_k$.

\paragraph{Hardness of shellability.}
In~\cite{goaoc-patak-patakova-tancer-wagner19} Goaoc, Pat\'ak, Pat\'akov\'a,
Tancer and Wagner proved that shellability is NP-hard. We state a corollary of
the main technical proposition from~\cite{goaoc-patak-patakova-tancer-wagner19}
in a way convenient for us. In the statement we use \emph{3-CNF formulas}, that
is, formulas in a conjunctive normal form where each clause contains three
literals. We skip details, referring the reader to ~\cite{goaoc-patak-patakova-tancer-wagner19}, as we use 3-CNF formulas only implicitly. 
We only need the fact that the decision problem whether a 3-CNF formula
is satisfiable is a well known NP-hard problem, known as
\emph{3-satisfiability}.

\begin{theorem}[Essentially Proposition~8 from~\cite{goaoc-patak-patakova-tancer-wagner19}]
\label{t:shell_hard}
  There is a polynomial time algorithm that produces from a given 3-CNF formula
  $\phi$ with $t$ variables a pure connected 2-dimensional complex $K_\phi$ with
  $\tilde\chi(K_\phi) = t$ such that the
  following statements are equivalent:

  \begin{itemize}
    \item[(i)] The formula $\phi$ is satisfiable.
    \item[(ii)] The second barycentric subdivision $\sd^2 K_\phi$ is shellable.
    \item[(ii')] The third barycentric subdivision $\sd^3 K_\phi$ is shellable.
    \item[(ii'')] The forth barycentric subdivision $\sd^4 K_\phi$ is shellable.
    \item[(iii)] The complex $K_\phi$ is collapsible after removing some $t$ triangles.
    \item[(iii')] The barycentric subdivision $\sd K_\phi$ is collapsible after removing some $t$ triangles.
    \item[(iii'')] The second barycentric subdivision $\sd^2 K_\phi$ is collapsible after removing some $t$ triangles.
  \end{itemize}
\end{theorem}

We are really interested only in the items (i), (ii) and (iii''). The remaining
items are auxiliary for explaining the proof.

Because Proposition~8 from~\cite{goaoc-patak-patakova-tancer-wagner19} is
not formulated exactly this way, we briefly explain how Theorem~\ref{t:shell_hard}
follows from~\cite{goaoc-patak-patakova-tancer-wagner19}: The construction of
$K_\phi$ is according to
~\cite[Proposition~8]{goaoc-patak-patakova-tancer-wagner19}. The fact that the
number of variables of $\phi$ equals $\tilde\chi(K_\phi)$ is the content
of~\cite[Proposition~12]{goaoc-patak-patakova-tancer-wagner19}. Then the
statements (i), (ii), (ii'), (iii) and (iii') of Theorem~\ref{t:shell_hard} are explicitly 
stated as equivalent statements in the (joint) proof of Theorems~4 and~5
in~\cite{goaoc-patak-patakova-tancer-wagner19}. It remains to argue that (ii'')
and (iii'') are equivalent as well. The proof of Theorems~4 and~5
in~\cite{goaoc-patak-patakova-tancer-wagner19} contains in particular
implications (ii) $\Rightarrow$ (ii') $\Rightarrow$ (iii') $\Rightarrow$ (i).
The implications (ii') $\Rightarrow$ (ii'') $\Rightarrow$ (iii'') $\Rightarrow$
(i) work with the exactly same reasoning, which proves the equivalence.
Finally, it is possible to check that $K_\phi$ is connected directly from the
construction in~\cite{goaoc-patak-patakova-tancer-wagner19}. Alternatively,
Skotnica and Tancer proved~\cite[Appendix~A]{skotnica-tancer23} that
$K_\phi$ from exactly this construction is homotopy equivalent to the wedge of
spheres. This also implies that $K_\phi$ is connected.

In our proof of Theorem~\ref{t:main}, we use Theorem~\ref{t:shell_hard}
with $L_\phi = \sd^2 K_\phi$, extending it to the following
setting. 
%Namely we extend Theorem~\ref{t:shell_hard} to the
%following setting. 

\begin{theorem}
\label{t:auxiliary}
   There is a polynomial time algorithm that produces from a given 3-CNF formula
  $\phi$ with $t$ variables a pure 2-dimensional connected complex $L_\phi$ with $\tilde
  \chi(L_\phi) = t$ such that the
  following statements are equivalent:
  \begin{enumerate}[(i)]
    \item The formula $\phi$ is satisfiable.
    \item The complex $L_\phi$ is shellable.
    \item The complex $L_\phi$ is collapsible after removing some $t$ triangles.
    \item We have $\wsat(L_\phi^{(1)}, K_3) = n-1$ where $n$ is the number of
      vertices of $L_\phi$.
    \item The complex $L_\phi$ is collapsible after removing some number of
      triangles.
  \end{enumerate}
\end{theorem}

The proof of Theorem~\ref{t:auxiliary} is given in the next section. Theorem~\ref{t:main} follows immediately from the equivalence of (i) and
(iv) in Theorem~\ref{t:auxiliary}.

\begin{proof}[Proof of Theorem~\ref{t:main}]
  The equivalence of (i) and (iv) provides a polynomial time 
  reduction from 3-satisfiability to determining whether $\wsat(F,
  K_3) = n-1$. (Note that the graph $L_\phi^{(1)}$ can be constructed from $L_\phi$ in
  polynomial time.)  Given that 3-satisfiability is NP-hard, it follows that the
  latter problem is NP-hard as well.
\end{proof}

Finally, we remark that the items (ii), (iii) and (v) in
Theorem~\ref{t:auxiliary} are again only auxiliary in order to show conveniently the
equivalence of (i) and (iv).

\section{The proof of Theorem~\ref{t:auxiliary}}

The aim of this section is to prove Theorem~\ref{t:auxiliary}, completing the proof of Theorem~\ref{t:main}.

As stated earlier, given a 3-CNF formula $\phi$ we take $K_\phi$ from
Theorem~\ref{t:shell_hard} and set $L_\phi := \sd^2 K_\phi$. Given that
$K_\phi$ is $2$-dimensional, the complexity of $L_\phi$ grows only by a
constant factor when compared with $K_\phi$, and so $L_\phi$ can be constructed in
polynomial time in the size of $K_\phi$, hence in polynomial time in the size of
$\phi$. The complex $L_\phi$ is connected because $K_\phi$ is connected. We also remark that $\tilde \chi (L_\phi) = \tilde \chi (K_\phi) = t$
because a complex and its barycentric subdivision have the same reduced Euler
characteristic.\footnote{A complex and its barycentric subdivision are homeomorphic (see, for example, \cite[\S 15]{munkres84}) and
the (reduced) Euler characteristic is an invariant under homeomorphism, even
under homotopy~\cite[Theorem~2.44]{hatcher02}. Alternatively, in the case of
$2$-complexes, one can check
directly from the definition of the barycentric subdivision that given a
2-complex $K$ with $n$ vertices, $m$ edges and $t$ triangles, $\sd K$ contains
$n+m+t$ vertices, $2m + 6t$ edges and $6t$ triangles.}
The items (i), (ii) and (iii) of Theorem~\ref{t:auxiliary} are
equivalent due to Theorem~\ref{t:shell_hard}. We will now show implications (ii)
$\Rightarrow$ (iv) $\Rightarrow$ (v) $\Rightarrow$ (iii), completing the proof.

\begin{proof}[Proof of (ii) $\Rightarrow$ (iv)]
  Consider a \emph{shelling} $\vartheta_1, \dots, \vartheta_m$ of $L_\phi$,
  that is, a
  sequence of all facets of $L_\phi$ witnessing that $L_\phi$ is shellable.
  Given that $L_\phi$ is pure 2-dimensional, all facets are triangles. 
  
  We construct a spanning tree $G$ in the 1-skeleton $L_\phi^{(1)}$ in the
  following way. First we set $G_1$ inside $\vartheta_1$ to contain
  all three vertices and two arbitrarily chosen edges. Next, for $i \in \{2,
  \dots, m\}$, we inductively
  assume that we have a spanning tree $G_{i-1}$ of $L_\phi[\vartheta_1, \dots,
  \vartheta_{i-1}]^{(1)}$ and we construct a spanning tree $G_i$ of $L_\phi[\vartheta_1, \dots,
  \vartheta_{i}]^{(1)}$; see Figure~\ref{f:add_theta_i} for an illustration. We distinguish two cases. If $\vartheta_i$ meets the
  preceding triangles in two or three edges we set $G_i := G_{i-1}$. Note that 
  $L_\phi[\vartheta_1, \dots, \vartheta_{i-1}]$ and $L_\phi[\vartheta_1, \dots,
  \vartheta_{i-1}]$ have the same sets of vertices in this case. Thus $G_i$ is
  indeed a spanning tree. If $\vartheta_i$ meets the preceding triangles in a
  single edge, then we add to $G_{i-1}$ the new vertex of $\vartheta_i$ (i.e.,
  the vertex not contained in $L_\phi[\vartheta_1, \dots, \vartheta_{i-1}]$)
  and one edge inside $\vartheta_i$ containing this vertex. Other cases are not
  possible because $\vartheta_1, \dots, \vartheta_m$ is a shelling. We set $G
  := G_m$.

  \begin{figure}
    \begin{center}
      \includegraphics{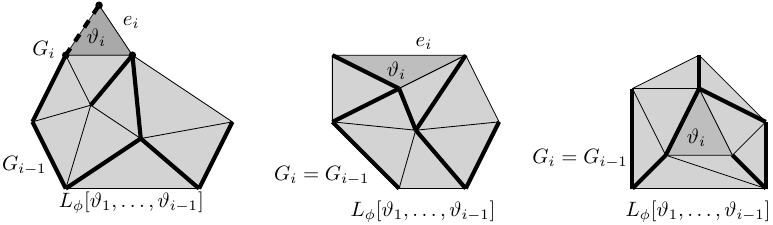}
      \caption{The figure displays three options how $\vartheta_i$ may meet
      $L_\phi[\vartheta_1,\dots,\vartheta_{i-1}]$ according to the number of
      shared edges. (The third displayed case is not fully realistic globally 
      because the displayed $L_\phi[\vartheta_1,\dots,\vartheta_{i-1}]$ is not
      shellable. Any realistic example unfortunately contains a nontrivial
      2-dimensional homology, which makes it harder to depict.)}
      \label{f:add_theta_i}
    \end{center}
  \end{figure}

  In order to finish the proof, we claim that $G$ is weakly $K_3$-saturated in
  $L_\phi^{(1)}$. Indeed, the saturating sequence follows the shelling. The
  first edge $e_1$ is the unique edge of $\vartheta_1$ not contained in $G_1$.
  This edge completes a copy of $K_3$ inside $\vartheta_1$ (more precisely
  $L_\phi[\vartheta_1]^{(1)}$).
  Next, for $i \in \{2, \dots, m\}$, we observe that $L_\phi[\vartheta_1, \dots,  \vartheta_{i}]^{(1)}$
  contains at most one more edge not contained in $G_i$ than $L_\phi[\vartheta_1, \dots,
  \vartheta_{i-1}]^{(1)}$. It is exactly one edge $e_i$ in the cases that $\vartheta_i$
  meets the preceding triangles in one or two edges; see again
  Figure~\ref{f:add_theta_i} for an example. (With a slight abuse of
  notation we denote the edge by $e_i$ though it may be not the $i$-th edge in the order, if some preceding edges are missing.) The edge $e_i$ again completes
  the copy of $K_3$ inside $\vartheta_i$, thereby inside $L_\phi[\vartheta_1,
  \dots,  \vartheta_{i}]^{(1)}$.
\end{proof}

\begin{proof}[Proof of (iv) $\Rightarrow$ (v)]
  The facts that $L_{\phi}$ is connected and $\wsat(L_\phi^{(1)}, K_3) = n-1$
  imply that there exists a
  spanning tree $G$ weakly $K_3$-saturated in $L_\phi^{(1)}$. We will show that
  $L_\phi^{(1)}$ collapses to $G$. This implies that $L_\phi^{(1)}$ is
  collapsible as any tree is collapsible.

  Let $e_1, \dots, e_m$ be a sequence of edges witnessing that $G$ is weakly
  $K_3$-saturated in $L_\phi^{(1)}$. For every such edge $e_i$ we fix a copy $J_i$ of $K_3$
  it creates. Now we crucially use that $L_\phi$ is a barycentric subdivision
  of another complex. It is well known and not hard to show (at least for
  $2$-complexes) that every copy of $K_3$ induces a triangle in $L_\phi$. (In
  general barycentric subdivisions are \emph{flag}, that is, every clique in
  the 1-skeleton induces a full simplex in the complex.) By $\vartheta_i$ we
  denote the triangle induced by $J_i$. We remark that the triangles
  $\vartheta_i$ are distinct because for $i < j$, $\vartheta_j$
  contains the edge $e_j$ while $\vartheta_i$ does not contain it.

 We set $L$ to be $L_\phi$ after removing all triangles that do not appear as
  $\vartheta_i$ for some $i \in [m]$. Now we perform elementary collapses on $L$ in
  the reverse order of $e_1, \dots, e_m$. That is, we first
  claim that $e_m$ is in a unique triangle $\vartheta_m$. This is indeed the case
  as the triangles $\vartheta_i$ with $i < m$ do not contain $e_m$. We perform
  an elementary collapse on $L$ removing $e_m$ and $\vartheta_m$. After
  performing this collapse we claim that $e_{m-1}$ is in a unique triangle
  $\vartheta_{m-1}$. This is indeed the case as $\vartheta_m$ has been already
  removed and the triangles $\vartheta_i$ with $i < m - 1$ do not contain
  $e_{m-1}$. We perform an elementary collapse on the intermediate complex
  removing $e_{m-1}$ and $\vartheta_{m-1}$. We continue this until we have collapsed $L$ to $G$ as required. (Note that every edge of $L$ outside $G$
  appears as some $e_i$ because $e_1, \dots, e_m$ witnesses that $G$ is weakly
  $K_3$ saturated in $L_\phi^{(1)}$ and $L_\phi^{(1)} = L^{(1)}$.)
\end{proof}

\begin{proof}[Proof of (v) $\Rightarrow$ (iii)]
	Let $L$ be a collapsible complex obtained from $L_\phi$ after removal
of $k$ triangles. We know that $\tilde \chi (L_\phi) = t$, thus $\tilde \chi
(L) = t - k$ immediately from the definition of the reduced Euler
  characteristic. Because collapses preserve the homotopy type (this is
  explained, e.g., in a slightly more general setting
  in~\cite[Chapter~3]{rourke-sanderson82}) and the (reduced)
  Euler characteristic is an invariant of the homotopy type (see~\cite[Theorem~2.44]{hatcher02}), we deduce $\tilde \chi
  (L) = \tilde \chi (pt)$ where $pt$ stands for a
  point.\footnote{Alternatively, one can check directly from the definition of
  an elementary collapse that it removes the same number of faces of odd and of even size.} However, $\tilde \chi
(pt) = 0$ (the empty set and the point itself contribute $-1$ and $1$,
respectively). Thus we deduce $k = t$, which proves (iii).  
\end{proof}

\section*{Acknowledgement}

We would like to thank Adam Rajsk\'y for helpful early discussions.  

\bibliographystyle{alpha}
\bibliography{wsathard}

\end{document}